\documentclass{article} [11 pt]

\usepackage{amsmath,amsthm}     
\usepackage{graphicx}     
\usepackage{hyperref} 
\usepackage{url}
\usepackage{amsfonts} 
\usepackage{amssymb}
\usepackage{siunitx}
\usepackage[mathscr]{euscript}
\usepackage{adjustbox}

\usepackage{multicol}
\usepackage{enumerate}
\usepackage[margin=1.0in,top=1.0in, bottom=1.8in]{geometry}

\ifpdf  \fi

\DeclareSymbolFont{msbm}{U}{msb}{m}{n}
\DeclareMathSymbol{\C}{\mathalpha}{msbm}{'103}
\DeclareMathSymbol{\R}{\mathalpha}{msbm}{'122}
\DeclareMathSymbol{\Q}{\mathalpha}{msbm}{'121}
\DeclareMathSymbol{\Z}{\mathalpha}{msbm}{'132}
\DeclareMathSymbol{\N}{\mathalpha}{msbm}{'116}
\DeclareMathSymbol{\K}{\mathalpha}{msbm}{'113}

\newtheorem*{definitions*}{Definitions} 
\newtheorem{lemma}{Lemma}
\newtheorem{cor}[lemma]{Corollary}
\newtheorem{theorem}[lemma]{Theorem}

\newtheorem{problem}{Problem}

\newcommand{\old}[1]{{}}

\newcommand{\bi}{\begin{itemize}}
\newcommand{\ei}{\end  {itemize}}
\newcommand{\bt}{\begin{tabbing}}
\newcommand{\et}{\end  {tabbing}}
\newcommand{\be}{\begin{enumerate}}
\newcommand{\ee}{\end  {enumerate}}

\newtheoremstyle{obs}
  {\topsep} 
  {0pt} 
  {} 
  {} 
  {\bfseries} 
  {.} 
  {.5em} 
  {} 

\theoremstyle{obs}

\makeatletter
\def\begin@lgo{\begin{minipage}{1in}\begin{tabbing}
        \quad\=\qquad\=\qquad\=\qquad\=\qquad\=\qquad\=\qquad\=\kill}
\def\end@lgo{\end{tabbing}\end{minipage}}

\makeatother

\makeatletter
\@ifundefined{abovecaptionskip}{\newlength\abovecaptionskip}
\long\def\@makecaption#1#2{
   \vskip \abovecaptionskip
   \setbox\@tempboxa\hbox{{\sf\footnotesize \textbf{#1.} #2}}
   \ifdim \wd\@tempboxa >\hsize         
       {\sf\footnotesize \textbf{#1.} #2\par}
     \else                              
       \hbox to\hsize{\hfil\box\@tempboxa\hfil}
   \fi}
\makeatother

\title{On a  Generalization of the Marriage Problem}
\author{
Jonathan Lenchner \thanks{IBM T. J. Watson Research Center, Yorktown Heights, NY USA; lenchner@us.ibm.com.}
}                    
\date{}

\begin{document}


\maketitle

\section*{Abstract}

We present a generalization of the marriage problem underlying Hall's famous Marriage Theorem to what we call the Symmetric Marriage Problem, a problem that can be thought of as a special case of Maximal Weighted Bipartite Matching.  We show that there is a solution to the Symmetric Marriage Problem if and only if a variation on Hall's Condition holds on each of the bipartitions. We prove both finite and infinite versions of this result and provide applications. 

\section{Introduction}

In the folklore interpretation of the famed Marriage Theorem due to Philip Hall \cite{HALL:1935} one is given a set $G$ of girls and a set $B$ of boys. Each girl identifies a list of boys that she would willingly marry. Each boy is willing to marry any girl that is willing to marry him, or no one at all. The theorem states that as long as for any subset of the girls, the size of the union of the associated lists of boys is at least as large as the number of girls, then there is a matching (in other words, an injective function) of girls to boys that makes everyone happy. 

In this paper we consider a generalization of the marriage problem in which not just the girls but also the boys may make lists of who they are willing to marry. Crucially, it is also permitted that a girl or boy not make a list, which, like in the classical marriage problem, means they are willing to marry anyone willing to marry them, or no one at all. We call this variation of the marriage problem the Symmetric Marriage Problem or SMP because the girls and boys are treated equally. Further, the SMP generalizes the problem behind Hall's Theorem, which we refer to as the Classical Marriage Problem or CMP, because if we refer to $G_L \subset G$ as the subset of girls who make lists, and $B_L \subset B$ as the subset of boys who make lists, then a CMP is just an SMP with $G_L = G$ and $B_L = \varnothing$.

%
%

The remainder of this paper is structured as follows. In the next section we give a formal statement of Hall's Theorem and provide a variety of applications. In the section that follows we state and prove the finite Symmetric Marriage Theorem and provide applications of it. We then prove an infinite version of the Symmetric Marriage Theorem and finish up with some additional considerations and questions.

\section{Hall's Marriage Theorem and the Classical Marriage Problem}

In 1935 Philip Hall proved the following celebrated theorem \cite{HALL:1935}:


\begin{theorem} \label{thm:hmt_classical} \textbf{Hall's Marriage Theorem} 
Let $\{B_g\}_{g \in G}$ be a finite collection of subsets of a finite set $B$. If for any $G' \subset G$, $|G'| \leq |\cup{\{B_g\}_{g \in G'}}|$ then there is an injective function $T:G \rightarrow B$ such that $T(g) \in B_g$ for all $g \in G$. Conversely, given a collection $\{B_g\}_{g \in G}$ of subsets of $B$, if a function $T: G \rightarrow B$ exists such that $T(g) \in B_g$, then for any $G' \subset G$, $|G'| \leq |\cup{\{B_g\}_{g \in G'}}|$.
\end{theorem}

As noted in the introduction, the folklore way to think about the Marriage Theorem is that you are given a set $G$ of girls and a set $B$ of boys. Each girl identifies a set of boys that she would willingly marry. 
Additionally, one supposes that each boy is willing to marry any girl that is willing to marry him, or no one at all. Then as long as for any subset, $G'$, of girls, the size of the union of the set of boys they would like to marry is as large as $|G'|$, then there is a one-to-one pairing of girls to boys satisfying all of the girls' (and hence everyone's) wishes. If $|B| > |G|$ then the map is injective but not surjective. The converse is obvious; for a pairing of girls to boys to exist there cannot be a subset of girls, the union of whose acceptable pairings is smaller than the number of girls in the subset.

Let us follow convention and refer to the criterion of Theorem \ref{thm:hmt_classical}, that for any $G' \subset G$, $|G'| \leq |\cup{\{B_g\}_{g \in G'}}|$ as \textbf{Hall's Condition}, or alternatively, the \textbf{Marriage Condition}. 

\begin{figure}[h]
\centerline{\scalebox{0.20}{\includegraphics{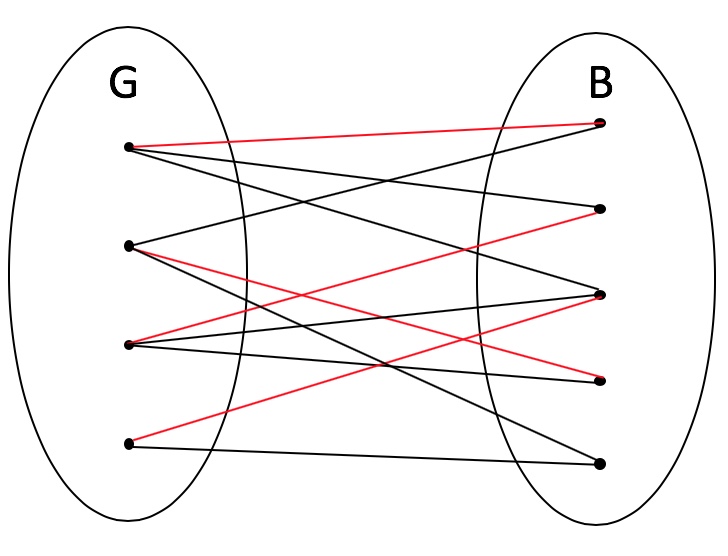}}}
\caption{The folklore view of the Marriage Theorem. The matching in red is one that makes everyone happy.} \label{fig:cmp}
\end{figure}

We shall call the problem associated with the Marriage Theorem, the Classical Marriage Problem or CMP, though it is perhaps more often referred to simply as bipartite matching. We will refer to a CMP instance $\mathscr{C}$ using the triple $\mathscr{C} = (G, B, \{B_g\}_{g \in G})$. Hall's Theorem does not give an effective (in other words, polynomial time) means of determining if a given bipartite graph has a maximum matching that completely covers the set $G$. To make this determination using the theorem one would have to test all subsets $G' \subset G$ and consider the size of the associated sets  $\cup_{g \in G'} B_g$. However one can find a maximum matching in a bipartite graph by viewing the problem as a flow problem. Connect all vertices in $G$ to a source node $s$ and all vertices in $B$ to a sink node $t$, and have the flow go from $\{s\} \rightarrow G \rightarrow B \rightarrow \{t\}$, with all edges given capacity of $1$ unit.  Then, applying, for example, the Edmonds-Karp implementation of the Ford-Fulkerson method \cite{CORMEN:1999, EDMONDS:1972}, and observing that the flow associated with any augmenting path is always integral, we see that a max-flow corresponds to a maximal matching, and hence a solution to the CMP can be found in the runtime of the Edmonds-Karp algorithm, or $O(V E^2)$, where $V$ and $E$ denote, respectively, the number of vertices and edges in the graph. Faster algorithms are known; the fastest known algorithm being that due to Hopcroft and Karp \cite{H-K:1973}, which runs in $O(\sqrt{V}E)$ time.

The magic of Hall's Theorem is that sometimes one can inspect the conditions of the problem and determine that Hall's Condition applies and hence the problem must have a solution, as is famously the case for a variety of facility location problems (see, e.g., \cite{CYGAN:2012}), a number of combinatorial problems and even a problem in group theory. The first problem below is from the theory of tournaments while the second problem comes from the field of combinatorial designs \cite{PUTNAM:2012, brilliant:hall, brilliant:hall-apps}.

Consider a round robin tournament in which each of $2n$ teams plays a series of head-to-head matches against each other team over the course of $2n-1$ sessions. Each session consists of $n$ head-to-head matches with each team participating in one match. Each match is decisive. In other words, each match produces a winner and a loser. Is it then possible to pick a unique winner from each of the rounds?

The answer is yes. On the $i$th day, let $W_i$ denote the set of winning teams on that day. 
If it were not possible to pick a set of unique winners then we must have a collection of $k \leq 2n$ sets $\{W_{i_j}\}$, the size of the union of whose elements is less than $k$. This means that there must be a team that has lost each of these $k$ games (otherwise $\cup \{W_{i_j}\} = 2n$). But then, since the tournament is a round-robin, this team has lost to $k$ \textit{distinct} other teams and so $\cup \{W_{i_j}\} \geq k$, a contradiction.  It follows that the set $\{W_i\}$ satisfies Hall's Condition and hence it must always be possible to pick a unique winner in each round of the tournament.

In the second problem we consider a $2n$ x $2n$ chess board where we place $n$ rooks in each row and $n$ rooks in each column. We claim that it is always possible to find a sub-collection of the rooks of size $n$ 
with a unique rook in each row and column (i.e., such that no two rooks attack one another). If there were no such sub-collection then, by Hall's Theorem, we would have to have a sub-collection of $K$ rows, the union of whose rooks occupy columns of size less than $K$. Since any one row has rooks in $n$  columns we must have $K = n + k$ for some $1 \leq k \leq n - 1$. Note that if all $2n$ rows contained rooks that were confined to fewer than $2n$ columns then there would be some column with no rooks, which is impossible. But, analogously, if there are $n + k$ rows with rooks confined to fewer than $n + k$ columns, then the remaining columns must have rooks that are confined to fewer than $n$ rows, a contradiction. 

For an additional application in the field of combinatorial designs, showing that one can always extend a Latin rectangle to a Latin square, see \cite{brilliant:hall-apps}. In group theory it is possible to use Hall's Theorem to show that if $H$ is a subgroup of a finite group $G$ then there is a set of representatives $\{g_i\} \subset G$ such that $\{g_iH\}$ cycles through the left cosets of $H$ without repetition, while $\{Hg_i\}$ simultaneously cycles through the right cosets of $H$ without repetition. For details see \cite{stack:cosets}.

\section{Main Result}


In addition to not being particularly politically correct, the Classical Marriage Problem is not particularly fair to the boys. They are assumed to be willing to go along with whatever the girls wish and not have preferences of their own. A natural generalization of the Classical Marriage Problem or CMP is to suppose that not just the girls can make lists of who they are willing  to marry, but also the boys. Moreover, to make the problem a complete generalization of the CMP, we allow for the possibility that either girls or boys may not make lists -- meaning that they are happy to go along with any arrangement, i.e., be married to anyone who wants to marry them or not be married at all. We allow that either a girl or boy declare that they wish not to be married at all, but in this case we delete them from anyone else's requirements and remove them from the problem. If, as a result, someone ends up with an empty list of potential partners their requirements cannot be met and the problem is unsatisfiable. Thus, henceforth, we assume that a boy or girl has a non-empty list of members of the opposite sex that they are willing to marry, or, if they are a girl, they are willing to marry any boy (or no one) and if they are a boy they are willing to marry any girl or no one. As we did for the CMP, for a given $g \in G$ we let $B_g$ refer to the set of boys $g$ is willing to marry, and analogously, for $b \in B$ we let $G_b$ refer to the set of girls $b$ is willing to marry. If $B_g = \varnothing$ then $g$ is assumed to be willing to marry any boy or no one, and analogously for $G_b$. Thus $B_g = \varnothing$ or $G_b = \varnothing$ have special meaning. If $B_g \neq \varnothing$ we say that $g$ \textit{has a list}, that list being $B_g$. Analogously if $G_b \neq \varnothing$ we say that $b$ \textit{has a list}, that list being $G_b$. 


%
%
Let us call this problem the \textbf{Symmetric Marriage Problem} or \textbf{SMP}. The CMP is a special case of the SMP where each girl makes a list but no boy does. We refer to an SMP instance using the $4$-tuple $\mathscr{S} = (G, B, \{B_g\}_{g \in G}, \{G_b\}_{b \in B})$. Further, given an SMP instance, $\mathscr{S}$, let $G_L \subset G$ denote the set of girls with lists and $B_L \subset B$ denote the set of boys with lists.  A formal statement of the SMP is as follows:

\begin{problem} \label{prob:formal_smp} \textbf{Symmetric Marriage Problem (SMP)} Let $G, B$ be two finite sets and to each $g \in G$ associate a set $B_g \subset B$ and to each $b \in B$ associate a set $G_b \subset G$. Does there exist an injective partial function $P:G \rightarrow B$, with $G_L \subset \textrm{Domain}(P), B_L \subset \textrm{Range}(P)$, such that $P(g) \in B_g~\forall g \in G_L$, and $P^{-1}(b) \subset G_b~\forall b \in B_L$?
\end{problem}

\smallskip

Before proceeding let us remark on a variant of the SMP, which we will call the \textbf{Baby SMP}. The Baby SMP is an alternative, though slightly weaker way in which the Classical Marriage Problem is sometimes framed, namely:

\begin{problem} \label{prob:baby_smp} \textbf{Baby SMP} Suppose one has a finite set $G$ of girls and a finite set $B$ of boys, and that each girl has been introduced to some non-empty subset of the boys and vice versa. Is there a matching of girls to boys such that each is paired with someone they have been introduced to? 

\noindent \textbf{Baby SMP - Formal Version}: Let $G$ and $B$ be two finite sets. To each $g \in G$ define a set $B_g \subset B,~B_g \neq \varnothing$, and analogously, for each $b \in B$ define $G_b \subset G$, $G_b \neq \varnothing$. Further, suppose that $b \in B_g \leftrightarrow g \in G_b$. Does there exist a bijective map $T:G \rightarrow B$ such that $T(g) \in B_g~\forall g \in G$ and $T^{-1}(b) \in G_b~\forall b \in B$? \end{problem}

For there to be a solution to a given Baby SMP instance we clearly must have $|G| = |B|$. Thus if each girl is willing to marry each boy she has been introduced to, we can reframe any Baby SMP as a CMP instance. We could have reframed the problem as if the boys had no special preferences (i.e., they would be willing to marry anyone willing to marry them, or no one at all) but by the symmetry of the relation of one person having been introduced to another person, a boy will automatically be satisfied with any girl who is satisfied with him in a potential solution and so there is no harm in introducing the more restrictive condition.


\bigskip

Now let us return to the full SMP (Problem \ref{prob:formal_smp}). Consider an SMP instance $\mathscr{S} = (G, B, \{B_g\}_{g \in G}, \{G_b\}_{b \in B})$. Recall that $G_L$ designates the set of girls with lists and $B_L$ the set of boys with lists.  Pare down the list of each $g \in G_L$ such that $b \in B_g \cap B_L$ only if $g \in G_b$, in other words, only if $b$ and $g$ are on each other's lists -- in this case we will say that $b$ and $g$ are \textbf{list compatible}. There is no loss in generality in doing so, since if $g \in G_L$ and $b \in B_g \cap B_L$ then we can only pair $g$ with $b$ if the two elements are list-compatible. Refer to the pared down set of lists as $\{B^*_g\}_{g \in G_L}$. Analogously, pare down the list of each $b \in B_L$ such that  $g \in G_b \cap G_L$ only if $g$ and $b$ are list compatible and refer to the pared down set of lists as $\{G^*_b\}_{b \in B_L}$.  

\begin{theorem} \label{thm:main} Consider the SMP $\mathscr{S} = (G, B, \{B_g\}_{g \in G}, \{G_b\}_{b \in B})$. $\mathscr{S}$ is solvable iff the respective CMPs $\mathscr{C}_G = (G_L, B, \{B^*_g\}_{g \in G_L})$ and $\mathscr{C}_B = (B_L, G, \{G^*_b\}_{b \in B_L})$ are both solvable.
\end{theorem}

\begin{proof} 
The only if part of the theorem is obvious so we focus our attention on showing that $\mathscr{S}$ is solvable so long as both $\mathscr{C}_G$ and $\mathscr{C}_B$ are solvable. For this purpose, given an SMP instance $\mathscr{S} = (G, B, \{B_g\}_{g \in G}, \{G_b\}_{b \in B})$, create a graph consisting of four sets of nodes. Two of the sets will be the elements of $G$ and $B$, which we will denote by these same letters. In addition, create a node $L_g$ for each $g \in G_L$, with $L_g$ representing the set of boys on girl $g$'s list that are list compatible with $g$\footnote{Note that $L_g \subset B^*_g$, with the difference that $B^*_g$ may additionally contain elements in $B_g \setminus B_L$.}. Denote the set $\{L_g\}_{g \in G_L}$ by $L_G$. The fourth group of nodes then consists of a node $L_b$ for each $b \in B_L$, with $L_b$ representing the set of \textit{girls} on boy $b$'s list that are list-compatible with him. Denote the set $\{L_b\}_{b \in B_L}$ by $L_B$.

From the SMP $\mathscr{S}$, and the four sets of nodes as just described, we form a graph, which we shall denote by $G^*(\mathscr{S})$, where we connect a vertex $g \in G$ to a vertex $b \in B$ iff either (i) $b$ is on $g$'s list and $b$ has no list, or (ii) $g$ is on $b$'s list and $g$ has no list. Next if $g \in G_L$ and $b \in B_L$ and $g$ and $b$ are list-compatible we connect $g$ to $L_b$ and $b$ to $L_g$. A sample SMP instance $\mathscr{S}$ with associated graph $G^*(\mathscr{S})$ is depicted in Figure \ref{fig:4_partitions}.
\begin{figure}[h] 
\centerline{\scalebox{0.33}{\includegraphics{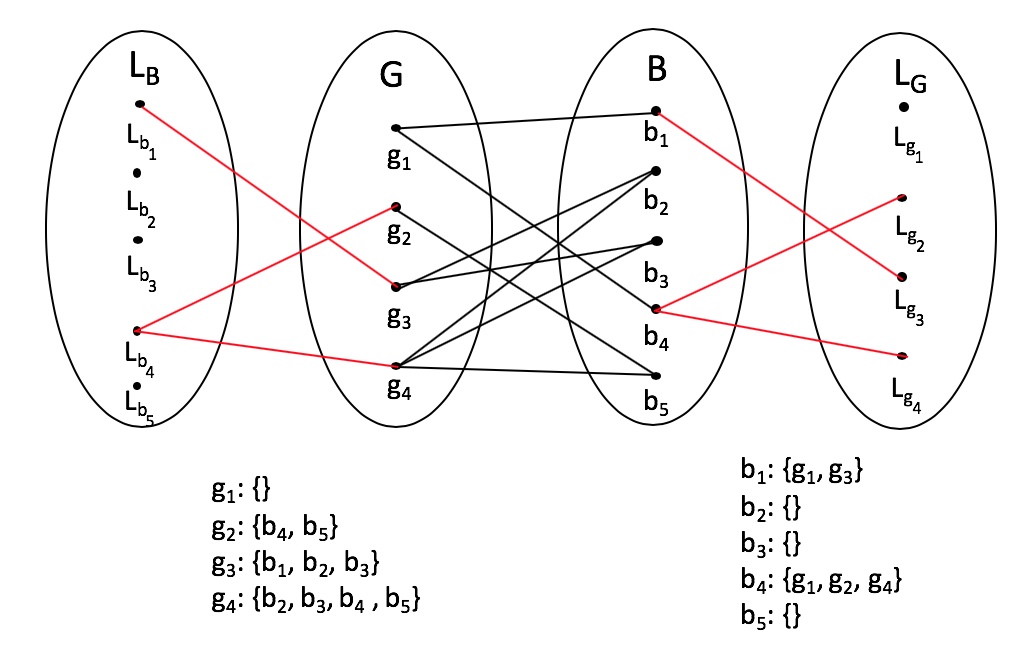}}}
\caption{A sample SMP instance $\mathscr{S}$ together with the associated graph $G^*(\mathscr{S})$. The lists of the girls are given on the bottom left and the lists of the boys are given on the bottom right. Edges in $G^*(\mathscr{S})$ between vertices in $G$ and vertices in $B$ are drawn in black, while edges incorporating vertices from $L_B$ or $L_G$ are drawn in red for clarity.} \label{fig:4_partitions}
\end{figure} 

%
Note that a solution to $\mathscr{C}_G$ corresponds to a matching of size $|G_L|$ in $G^*(\mathscr{S})$ touching each $g \in G_L$ but no other elements of $G$, while a solution to $\mathscr{C}_B$ corresponds to a matching of size $|B_L|$ in $G^*(\mathscr{S})$ touching each $b \in B_L$ but no other elements of $B$. Note, moreover, that these matchings are vertex-wise disjoint. Let us suppose then that we have a solution to both $\mathscr{C}_G$ and $\mathscr{C}_B$, and hence a combined matching $M$, which is of size size $|G_L| + |B_L|$.

We describe how to turn $M$ into a solution to $\mathscr{S}$. If $M$ contains edges $\{(g_{\alpha_i},b_{\beta_i}): g_{\alpha_i} \in G, b_{\beta_i} \in B\}$, and any time $M$ contains an edge $(L_{b_j}, g_i)$ it contains the corresponding edge $(b_j, L_{g_i})$ then we can convert such a matching into a partial function $P$ such that $P(g_{\alpha_i}) = b_{\beta_i}$ if $(g_{\alpha_i},b_{\beta_i}) \in M$ and $P(g_i) = b_j$ if both $(L_{b_j}, g_i), (b_j, L_{g_i}) \in M$. It is easy to check that $P$ is then a solution to $\mathscr{S}$. Thus if $M$ does not immediately provide a solution to $\mathscr{S}$ it must be the case that are some $K > 0$ edges with $(L_{b_j}, g_i) \in M$ but $(b_j,  L_{g_i}) \notin M$, or vice versa. Call these $K$ edges, $\{(L_{b_j}, g_i) \in M: (b_j, L_{g_i}) \notin M\} \cup \{(b_j, L_{g_i}) \in M: (L_{b_j}, g_i) \notin M\}$, the \textit{mismatched edges}. Under these circumstances we show how to alter the matching $M$ to get a new matching $M'$, still of size  $|G_L| + |B_L|$, but with fewer than $K$ mismatched edges. We can run this process at most $K$ times in succession to get a matching without mismatched edges, thereby turning $M$ into a solution to $\mathscr{S}$ completing the proof.

Thus, without loss of generality, suppose we have $(L_{b_1}, g_1) \in M$, establishing that $g_1$ and $b_1$ are list-compatible, but $(b_1, L_{g_1}) \notin M$. If there is no edge in $M$ that includes $L_{g_1}$ then we can just replace whatever edge includes $b_1$ in $M$ with $(b_1, L_{g_1})$ and still have a maximum matching, but with fewer mismatched edges, which would complete the proof. Thus assume we have $(b_2, L_{g_1}) \in M$, for some element that we designate as $b_2$, with $b_2 \neq b_1$, establishing the list-compatibility of $g_1$ and $b_2$. Now suppose there were no edge in $M$ containing $L_{b_2}$. Then we could swap $(L_{b_1}, g_1)$ in $M$ for $(L_{b_2}, g_1)$, keep the maximum matching and again reduce the number of mismatched edges, completing the argument.  Thus, for some element that we denote by $g_2$, with $g_2 \neq g_1$, we may suppose we have $(L_{b_2}, g_2) \in M$.  As earlier, we must have some edge connected to $L_{g_2}$. If it is $b_1$ then we can replace $(b_2, L_{g_1}), (b_1, L_{g_2})$ with $(b_1, L_{g_1}), (b_2, L_{g_2})$, reducing the number of mismatched edges. Hence we must have some element $b_3 \notin \{b_1, b_2\}$ with $(b_3, L_{g_2}) \in M$. 

The general inductive step proceeds with an analogous back and forth / right hand side-left hand side argument. For the right hand side argument we assume we have inductively established edges $\{(L_{b_i}, g_i)\}_{i=1}^k \cup (b_{i+1}, L_{g_i})\}_{i=1}^{k-1} \subset M$. Then, if $L_{g_k}$ is not covered by any edge in $M$ we can swap $\{(b_{i+1}, L_{g_i})\}_{i=1}^{k-1}$ for $\{b_i, L_{g_i})\}_{i=2}^k$ in $M$, freeing up $L_{g_1}$, so that we can then also swap whatever edge includes $b_1$ in $M$ for $(b_1, L_{g_1})$. This swapping repairs all mismatched edges thus-far encountered and preserves the fact that $M$ is a maximum matching. Thus, in an effort to break our inductive argument, we may assume that $L_{g_k}$ is covered by an edge in $M$. If the edge is $(b_1, L_{g_k})$, as in Figure \ref{fig:mismatched_edges}, 
\begin{figure}[h] 
\centerline{\scalebox{0.25}{\includegraphics{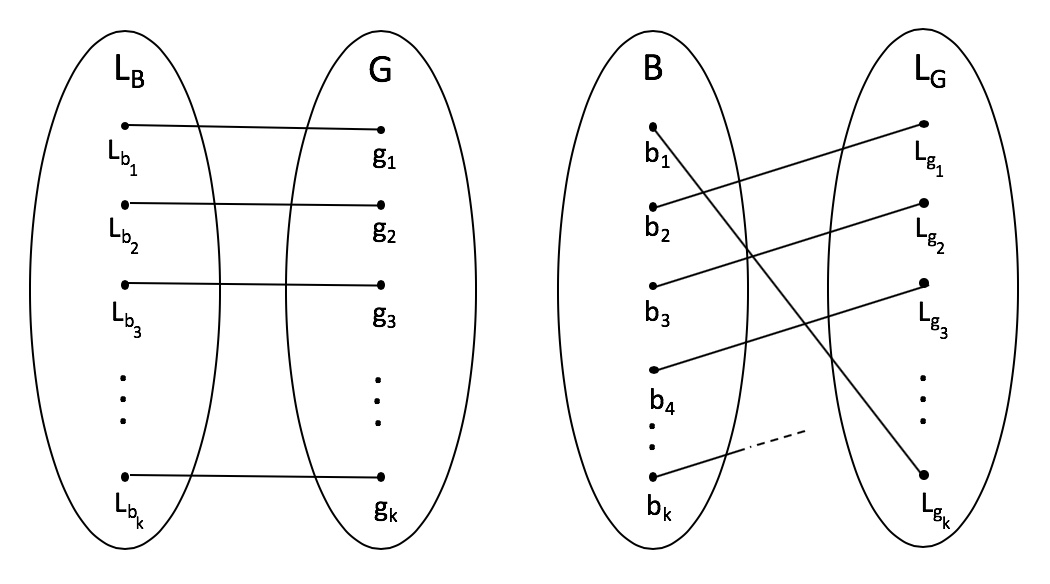}}}
\caption{Mismatched edges $\{(L_{b_i}), g_i\}_{i=1}^k \cup \{(b_{i+1}, L_{g_i})\}_{i=1}^{k-1} \cup \{(b_1, L_{g_k})\}$, as in the proof of Theorem \ref{thm:main}.} \label{fig:mismatched_edges}
\end{figure} 
then we can replace $\{(b_{i+1}, L_{g_i})\}_{i=1}^{k-1} \cup \{(b_1, L_{g_k})\}$ with $\{(b_i, L_{g_i})\}_{i=1}^k$, thereby repairing all mismatched edges thus far encountered. It follows that if there is to be a problem with our inductive argument then there must be an additional element $b_{k+1} \notin \{b_1,...,b_k\}$ with $(b_{k+1}, L_{g_k}) \in M$.

For the left hand side argument, we assume that we have inductively established the existence of edges $\{(L_{b_i}, g_i)\}_{i=1}^k \cup \{(b_{i+1}, L_{g_i})\}_{i=1}^k$. If $L_{b_{k+1}}$ is not covered by an edge of $M$ then we can swap $\{(L_{b_i}, g_i)\}_{i=1}^k$ for $\{(L_{b_{i+1}}, g_i)\}_{i=1}^k$ in $M$ and fix all mismatched edges.  It follows that there must be an additional element $g_{k+1} \notin \{g_1,...,g_k\}$ with $(L_{b_{k+1}}, g_{k+1}) \in M$.

Thus, in both the left and right hand side cases we conclude that there must be additional elements, $b_{k+1}$ to pair with $L_{g_k}$ on the right hand side, and $g_{k+1}$ to pair with $L_{b_{k+1}}$ on the left hand side. But there were only $K < \infty$ mismatched edges so the process must terminate in at most this many steps (a right hand, then left hand, argument comprising one such ``step''). The repair of all mismatched edges can therefore be made, after which we form the partial function $P: G \rightarrow B$ with $P(g) = b$ iff $(g,b) \in M$ or $(L_b, g) \in M$ (and so too, $(b, L_g) \in M$). $P$ thereby provides a solution to the SMP instance $\mathscr{S} = (G, B, \{B_g\}_{g \in G}, \{G_b\}_{b \in B})$, completing the proof. 
\end{proof}

The following is then an immediate consequence of Theorem \ref{thm:main}:

\begin{cor} \textbf{Hall's Bi-criteria for the SMP} \label{cor:2} Consider the SMP $\mathscr{S} = (G, B, \{B_g\}_{g \in G}, \{G_b\}_{b \in B})$.  $\mathscr{S}$ is solvable iff for every $\beta \subset B_L$,
\begin{equation} \label{eqn:hall1}
	|\cup_{b \in \beta} G^*_b| \geq |\beta|,
\end{equation}
and for every $\gamma \subset G_L$,
\begin{equation} \label{eqn:hall2}
	|\cup_{g \in \gamma} G^*_g| \geq |\gamma|.
\end{equation}
\end{cor}

\bigskip

Just like with Hall's Theorem, Corollary \ref{cor:2} does not provide a polynomial time algorithm for solving the SMP. However, three observations are in order.  First, note that any SMP instance can be thought of as a \textit{maximum weight} bipartite matching problem.  Connect a vertex $g \in G$ with a vertex $b \in B$ via a weight $1$ edge if either $g \in G_L \land b \in B_g \setminus B_L$ or $b \in B_L \land g \in G_b \setminus G_L$. Additionally connect a vertex $g \in G$ with a vertex $b \in B$ via a weight $2$ edge if $g$ and $b$ are list compatible, i.e., $b \in B_g \land g \in G_b$.  A solution to this maximum weight matching problem of size $|G_L| + |B_L|$ is readily seen to be a solution to the associated SMP instance and so can be solved by the Hungarian Algorithm \cite{MUNKRES:1957}.  The Hungarian algorithm typically requires a complete graph and an equal number of vertices in each of the two partitions.  For this purpose, just add an additional vertices required to either $B$ or $G$ and connect any vertices that are not yet connected via edges of weight $0$. Various implementations of the Hungarian Algorithm run in time $O(V^3)$. 

The second observation is that the graph $G^*(\mathscr{S})$ considered in the proof of Theorem \ref{thm:main} is actually a bipartite graph, where the nodes in $G \cup L_G$ form one of the bipartitions and the nodes in $B \cup L_B$ form the other. Since finding a solution to the underlying SMP $\mathscr{S}$ is equivalent to finding a maximum matching in $G^*(\mathscr{S})$ and testing whether the cardinality of that matching is equal to $|G_L| + |B_L|$, such a solution can be found in time $O(\sqrt{V}E)$ as we have previously remarked. 

The third observation is that since $G^*(\mathscr{S})$ is actually a bipartite graph, and so we end up solving a bipartite matching problem, one might think that we could get away with just a one-sided application of Hall's Theorem. The problem, however, is that in the Classical Marriage Problem, not only are we looking for a maximum matching but we are looking to completely cover $G$ and hence the index set in Hall's Condition in Theorem \ref{thm:hmt_classical} is clear. In the case of the SMP, however, it is not evident, a priori, what that index set should be -- it may be $G_L$, it may be $G$, or something in between.

\medskip

Let us pause now to consider a toy application.  Suppose we have a $4n$ x $4n$ chess board and there is a ``row player'' and a ``column player." Further, suppose the row player picks at least $3n$ rows, and within each row places a $+1$ in at least $3n$ cells and a $-1$ in the remaining cells. In the remaining rows she places a $0$ in all cells. Further, to insure the ``$-1$''s are modestly well-balanced, she makes sure that not more than $n$ of these $-1$ entries appear in the same column. The column player makes analogous entries, but in each column. Call the rows where the row player entered non-zero values the ``non-zero rows'' and the columns where the column player entered non-zero values the ``non-zero columns.'' The question is whether one can pick cells with exactly one cell coming from each non-zero row and column, such that the sum of the two entries in each such cell is positive. The answer is yes, and the argument is an easy application of Corollary \ref{cor:2}. We associate the row indices with the elements of $G$ and the column indices with the elements of $B$. $g_i$ is then shorthand for the $i$th row and $b_j$ is shorthand for the $j$th column.  The elements of $G_L$ are the non-zero rows, and the associated elements of $B_{g_i}$ are the columns with indices associated with the cells in which the row player entered the number $+1$ in row $i$. $B_L$ and $G_{b_i}$ are defined analogously. A sum of $+2$ in cell $(i, j)$ corresponds to a case where row and column indices $g_i$ and $b_j$ are list compatible, while a sum of $+1$ means that either $b_j$ is on $g_i$'s list and $b_j$ has no list, or $g_i$ is on $b_j$'s list but $g_i$ has no list.

The arguments to see that the two Hall Criteria, (\ref{eqn:hall1}) and (\ref{eqn:hall2}) above, are satisfied are symmetrical, so we consider just condition (\ref{eqn:hall1}). If there is a Hall violator then there is a collection of $K$ non-zero rows with positive cells sums that are constrained to coming from cells in these rows and fewer than $K$ columns. Since a single non-zero row has at least $2n$ associated positive cell sums, a hypothetical Hall violator must be of size $K > 2n$. It follows that there must be a column with more than $2n$ non-positive cell sums. Such a column would have to be a non-zero column, since otherwise, by virtue of the distribution of the row player's $-1$s being ``modestly well-balanced'' and not appearing in more than $n$ cells in the same column, there would have to be at least $2n$ positive cell sums in such a column leaving no more than $2n$ non-positive cell sums. But now a non-zero column has at least $3n$ positive column values. This coupled with the fact that there can be at most $n$ negative row values, means there must be more than $2n$ positive cell sums, again leaving no more than $2n$ non-positive cell sums. The two Hall Criteria, (\ref{eqn:hall1}) and (\ref{eqn:hall2}), are therefore satisfied and we are guaranteed to find the needed cells with positive cell sums.

The most natural applications of Theorem \ref{thm:main} and Corollary \ref{cor:2} are probably to simultaneous assignment problems. For example, one may have a set of workers and a set of tasks, where the workers can take on just a single task at a time and all the tasks are single-person tasks.  Some of the tasks are mandatory and must get done, and some of the workers are paid and must be put to work, while some workers are volunteers. Not all workers can perform all tasks. Is there an assignment of workers to tasks such that all paid workers are put to work and all mandatory tasks are worked on?

\section{Infinite Versions}


The following infinite version of the Marriage Theorem is due to Ron Aharoni \cite{AHARONI:1984}:

\begin{theorem} \label{thm:infinite_marriage} 
Let $\mathscr{S}$ be a possibly infinite collection of finite subsets of a possibly infinite set $X$. If for any $\mathscr{W} \subset \mathscr{S}$, $|\mathscr{W}| \leq |\cup_{W \in \mathscr{W}} W|$ then there is an injective function $T:\mathscr{S} \rightarrow X$ such that $T(S) \in S$ for all $S \in \mathscr{S}$. Conversely, given a collection $\mathscr{S}$ of subsets of $X$, if such a function $T$ exists, then for any $\mathscr{W} \subset \mathscr{S}$, $|\mathscr{W}| \leq |\cup_{W \in \mathscr{W}} W|$.
\end{theorem}

The inequalities in Theorem \ref{thm:infinite_marriage} are taken in the sense of (possibly infinite) cardinalities. In other words, $|\mathscr{W}| \leq |\cup_{W \in \mathscr{W}} W|$ means there is an injective function from $\mathscr{W}$ to $\cup_{W \in \mathscr{W}} W$. 

\medskip

The group theorist Marshall Hall (no relation to Philip Hall) observed in \cite{MHALL:1986} that Theorem \ref{thm:infinite_marriage} does not hold if we allow $\mathscr{S}$ to contain \textit{infinite} subsets of $X$. If $X = \{1,2,3,...\}$ and $S_0 = X, S_1 = \{1\}, S_2 = \{2\},...$ with $\mathscr{S} = \{S_0, S_1,...\}$, then $\mathscr{S}$ satisfies the infinite version of Hall's Condition but there is no transversal $T:\mathscr{S} \rightarrow X$ such that $T(S_i) \in S_i$ for all $S_i \in \mathscr{S}$.

\medskip

Returning now to the SMP, the \textit{infinite} Symmetric Marriage Problem is precisely the same as the original SMP just with possibly infinite sets. In other words, in the instance $\mathscr{S} = (G, B, \{B_g\}_{g \in G}, \{G_b\}_{b \in B})$, the sets $G$ and $B$ may each be infinite, for any $g \in G, B_g$ may be infinite, and for any $b \in B, G_b$ may be infinite. As earlier a solution to $\mathscr{S}$ is an injective partial function $P:G \rightarrow B$ such that  $G_L \subset \textrm{Domain}(P), B_L \subset \textrm{Range}(P)$ and $P(g) \in B_g~\forall g \in G_L$ and $P^{-1}(b) \in G_b~\forall b \in B_L.$. Note that at this point we are \textit{not} constraining the individual sets $B_g$ or $G_b$ to be finite, as in the statement of Theorem \ref{thm:infinite_marriage}.

Given an infinite SMP $\mathscr{S}$ we obviously can't say anything about the runtime for finding such an $P$. However, recalling our earlier definition of $B^*_g$, for $g \in G_L$, as the set $\{b \in B_g: b \notin B_L \vee (b \in B_L \land g \in G_b)\}$ and analogously for $b \in B_L,~G^*_b = \{g \in G_b: g \notin G_L \vee (g \in G_L \land b \in B_g)\}$, then we do have the following:

\begin{theorem} \label{thm:infinite_smp}
Consider the possibly infinite SMP $\mathscr{S} = (G, B, \{B_g\}_{g \in G}, \{G_b\}_{b \in B})$. $\mathscr{S}$ is solvable iff the respective possibly infinite CMPs $\mathscr{C}_G = (G_L, B, \{B^*_g\}_{g \in G_L})$ and $\mathscr{C}_B = (B_L, G, \{G^*_b\}_{b \in B_L})$ are both solvable.
\end{theorem}

\begin{proof}
If $\mathscr{S}$ is solvable then by definition $\mathscr{C}_G$ and $\mathscr{C}_B$ are both solvable, so assume that $\mathscr{C}_G$ and $\mathscr{C}_B$ are each solvable. The argument is similar to the proof of Theorem \ref{thm:main}. We again consider the graph $G^*(\mathscr{S})$ consisting of the four groups of nodes $G, B, L_G, L_B$ (refer back to Figures \ref{fig:4_partitions} and \ref{fig:mismatched_edges}), where we recall that the nodes of $L_G$ are representatives of the lists $L_g = \{b \in B_g : g \in G_b\}$ for $g \in G_L$, and analogously, the nodes of $L_B$ are representatives of the lists $L_b = \{g \in G_b : b \in B_g\}$ for $b \in B_L$.

As in the proof of Theorem \ref{thm:main} we must show how to take the matchings that are given by the solutions to $\mathscr{C}_G$ and $\mathscr{C}_B$ and turn them into a valid solution to $\mathscr{S}$. We are therefore again in the business of repairing the mismatched edges as we had previously defined them. Starting with any mismatched edge, we either run through the analysis we previously provided, leading to a finite sequence that repairs any encountered mismatched edges, or we arrive at an infinite sequence of mismatches as in Figure \ref{fig:infinite_mismatch}. 
\begin{figure}[h] 
\centerline{\scalebox{0.23}{\includegraphics{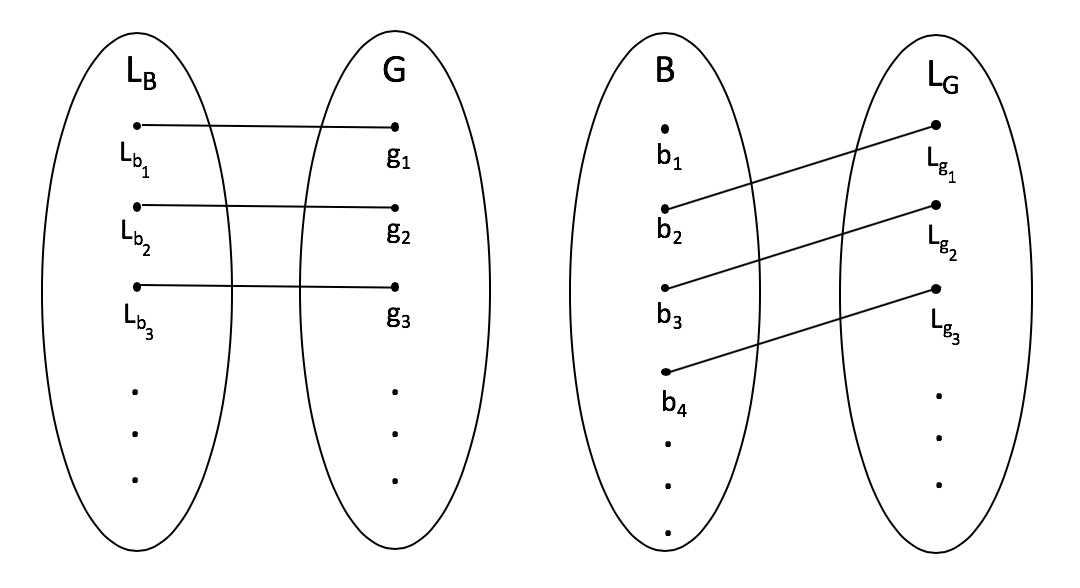}}}
\caption{An infinite sequence of mismatched edges.}. \label{fig:infinite_mismatch}
\end{figure} 
There are now two cases: (i) there is a first mismatch, i.e., either $(g_1, L_{b_1}) \in M$ but $(L_{g_1}, b_1) \notin M$, or we can back up from Figure \ref{fig:infinite_mismatch} to such a mismatch, or (ii) the example in Figure \ref{fig:infinite_mismatch} is part of a two-way infinite sequence of the form shown in Figure \ref{fig:2_way_infinite_mismatch}. 
\begin{figure}[h] 
\centerline{\scalebox{0.23}{\includegraphics{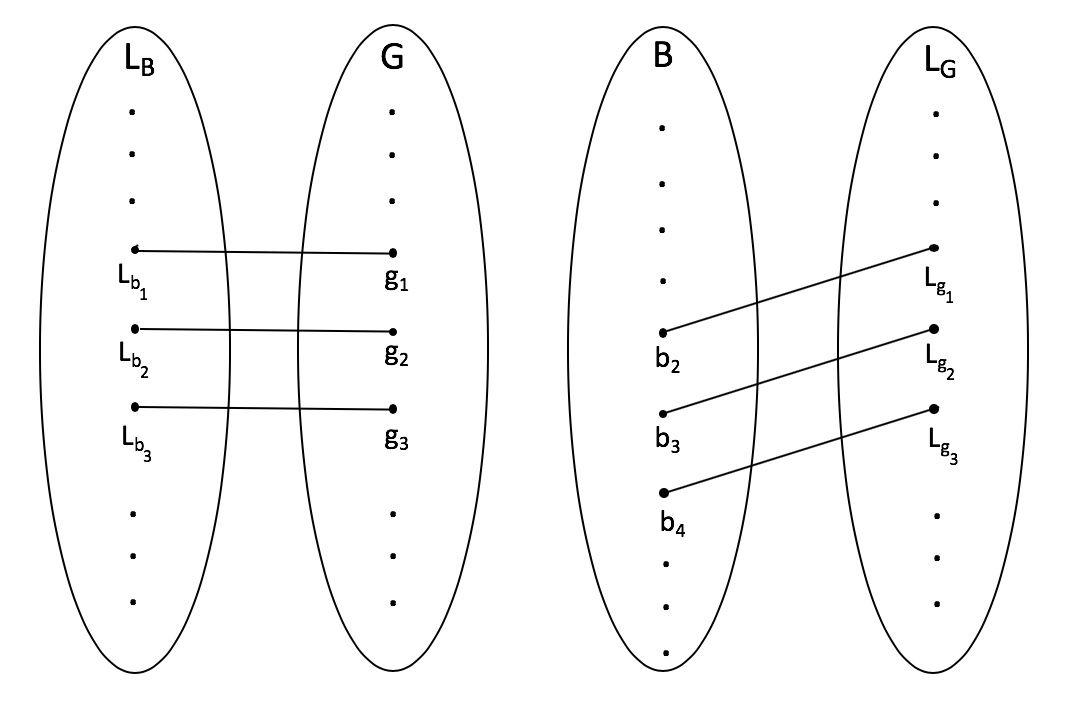}}}
\caption{A two-way infinite sequence of mismatched edges.}. \label{fig:2_way_infinite_mismatch}
\end{figure} 
In case (i) in which there is a first mismatch, assume without loss of generality that the first mismatch is $(g_1, L_{b_1}) \in M$ but $(L_{g_1}, b_1) \notin M$. 
In this case we can swap all edges $\{(g_i,L_{b_i})\}_{i=1}^\infty$ for $\{(g_i,L_{b_{i+1}})\}_{i=1}^\infty$, retaining the matching and repairing all mismatched edges. In case (ii) in which we have a two-way infinite sequence of mismatched edges, we can do the same, replacing each edge $(g_i,L_{b_i})$ with the associated edge $(g_i,L_{b_{i+1}})$, again retaining the maximum matching and fixing all the mismatched edges.

We note that all these maximal finite, maximal one-way infinite, and two-way infinite series of mismatched edges are non-overlapping so that all these operations can be performed independently, and, therefore, set theoretically (though not algorithmically) together. In fact, this process, despite working for arbitrary infinite cardinalities, does not require the Axiom of Choice. By repairing all mismatched edges we thereby turn a solution to $\mathscr{C}_G$ and $\mathscr{C}_B$ into a solution to $\mathscr{S}$ and the proof is therefore complete.
\end{proof}

Theorem \ref{thm:infinite_smp} together with the infinite version of Hall's Theorem (Theorem \ref{thm:infinite_marriage}) immediately imply the following. Note that it is here that we must introduce the assumptios that each $B_g$ and $G_b$ is finite.

\begin{cor} \textbf{Hall's Bi-criteria for the Infinite SMP}\label{cor:3} Consider the infinite SMP $\mathscr{S} = (G, B, \{B_g\}_{g \in G}, \{G_b\}_{b \in B})$ where each $B_g$ and each $G_b$ is \textit{finite}.  Then $\mathscr{S}$ is solvable iff for every $\beta \subset B_L$,
\begin{equation}
	|\cup_{b \in \beta} G^*_b| \geq |\beta|,
\end{equation}
and for every $\gamma \subset G_L$,
\begin{equation}
	|\cup_{g \in \gamma} G^*_g| \geq |\gamma|,
\end{equation}
\end{cor}

\section{Additional Considerations and Questions}

What happens if we try to obtain a Hall-like theorem for generic matching, i.e., what one might call the Non-Binary Marriage Problem, where anyone can opt to marry anyone else? If everyone makes a list of who they are willing to marry 
then we can test whether there is a matching that makes everyone happy simply by testing whether there is a perfect matching, now in a generic/not necessarily bipartite graph -- something we can do in polynomial time using the Edmonds blossom algorithm \cite{EDMONDS:1965}. The fastest deterministic algorithm to solve this problem is a variation on Edmonds' algorithm due to Micali and Vazirani \cite{MICALI:1980}. Moreover, the existence of such a matching can be guaranteed by the well known generalization of Hall's Theorem known as Tutte's Theorem \cite{BONDY:1976, LOVASZ:1986}. 

\begin{theorem} \textbf{Tutte's Theorem} 
A graph, $G = (V, E)$, has a perfect matching iff for every $U \subset V$, the subgraph induced by $V \setminus U$ has at most $|U|$ connected components with an odd number of vertices.
\end{theorem}

%
But what if some of the individuals in our non-binary marriage problem are allowed to specify that they are willing to marry anyone or no one? A solution is a matching that covers everyone with hard requirements.

Let us adopt notation for this problem that is akin to the notation we used in the binary case. Let $U$ denote the set of all individuals. We assume that $U$ is finite. As in the SMP we may assume that there is no one that wishes specifically not to be married to anyone.
Denote by $U_u$ the set of elements in $U$ that $u \in U$ is willing to marry. Let $U_L = \{u \in U: U_u \neq \varnothing\}$. $U_u = \varnothing$ has the special meaning that $u$ is willing to marry anyone or no one at all. A Non-Binary Marriage Problem (NBMP) instance is then given by the ordered pair $\mathscr{N} = (U, \{U_u\}_{u \in U})$.

Recall the for a given SMP $\mathscr{S} = (G, B, \{B_g\}_{g \in G}, \{G_b\}_{b \in B})$, $\mathscr{S}$ has a solution iff the two CMP instances $\mathscr{C}_G = (G_L, B, \{B^*_g\}_{g \in G_L})$ and $\mathscr{C}_B = (B_L, G, \{G^*_b\}_{b \in B_L})$ both have solutions. To this end, let us adopt the approach taken in the proof of Theorem \ref{thm:main} and consider two sets, $U$ and $L_U$. $U$ is, again, the set of all individuals, and $L_U$ contains a node, $L_u$, for each $u \in U_L$, and we think of $L_u$ as representing the set $\{u' \in U_L: u' \in U_u \land u \in U_{u'}\}$. We connect elements $u,u' \in U$ with an edge if either $u \in U_L \land u' \notin U_L \land u' \in U_u$ or  $u' \in U_L \land u \notin U_L \land u \in U_{u'}$. Further we connect an element $u \in U$ with an element $L_{u'} \in L_U$ iff $u \in L_{u'}$. Given an original problem instance $\mathscr{N} = (U, \{U_u\}_{u \in U})$, call the resulting graph $G^*(\mathscr{N})$. We would now like to argue that $\mathscr{N}$ is solvable iff there is a maximum matching in  $G^*(\mathscr{N})$ that covers $U_L$. However, the argument of Theorem \ref{thm:main} does not carry over. Consider the problem instance $U = \{u_1, u_2, u_3\}, L_{u_1} = \{u_2, u_3\}, L_{u_2} = \{u_1, u_3\},  L_{u_3} = \{u_1, u_2\}$. There are no elements of $U$ without lists so a matching must pair elements of $U$ with elements of $L_U$. One such maximum matching is $M = \{(u_1, L_{u_2}), (u_2, L_{u_3}), (u_3, L_{u_1})\}$. $M$ covers $U_L$. However, all of the edges of $M$ are mismatched in the sense that for each $(u_i, L_{u_j}) \in M$ there is not a corresponding edge $(u_j, L_{u_i}) \in M$.  Further, there is no way to repair all edges while still maintaining a matching covering $L_U$, since a ``repaired'' set of edges will have even cardinality. 

A more promising direction seems to be to find a suitable replacement for Tutte's condition and thereby a version of Tutte's Theorem that would apply to the NBMP so that we could conclude that under such a condition there must be a matching that covers $U_L$ in the graph $G^*(\mathscr{N})$.
As with the SMP one can turn an NBMP instance into a maximum weighted matching problem 
(though of course not a maximum weighted \textit{bipartite} matching problem)
and use known polynomial time algorithms to find a solution -- in fact the Micali and Vazirani algorithm mentioned earlier works for this problem as well \cite{MICALI:1980}.  

\medskip
%

Since the Symmetric Marriage Problem is a special case of Weighted Bipartite Matching it is natural to ask whether there are other, possibly more general, instances of Weighted Bipartite Matching for which a solution exists (in other words, a weighted matching reaches a certain threshold) iff a two-sided variation on Hall's Condition holds.

\section*{Acknowledgements}
The author gratefully acknowledges members of the Courant Geometry Seminar and especially Joe Malkevitch. Joe liked the Symmetric Marriage Problem enough to suggest the writing of this paper and provided valuable advise during the editing. 

\bibliographystyle{plain}
\bibliography{smp_soda_arxiv}

\end{document}